\documentclass{amsart}
\usepackage[utf8]{inputenc}
\usepackage[all]{xy}
\usepackage{upgreek,amsthm,bbm,bm}
\usepackage{amssymb,amstext,amsmath,amscd,amsthm,hyperref,tikz,tikz-cd,amsfonts,enumerate,graphicx,latexsym,color}
\usepackage[all]{xy}
\newtheorem{thm}{Theorem}[section]
\newtheorem{cor}[thm]{Corollary}
\newtheorem{lem}[thm]{Lemma}
\newtheorem{prop}[thm]{Proposition}

\newtheorem*{thm*}{Theorem}
\newtheorem*{cor*}{Corollary}
\theoremstyle{definition}

\newtheorem{defn}[thm]{Definition}
\newtheorem{rem}[thm]{Remark}
\newtheorem{ques}[thm]{Question}

\newtheorem*{conj*}{Conjecture}
\newtheorem{exam}[thm]{Example}

\newtheorem*{claim*}{Claim}
\newtheorem*{ques*}{Question}

\newtheorem{chunk}[thm]{\hspace*{-1.065ex}\bf}
\theoremstyle{remark}
\newtheorem*{ac}{Acknowledgments}
\numberwithin{equation}{thm}
\def\A{\mathcal{A}}
\def\aa{\mathfrak{a}}
\def\ann{\operatorname{Ann}}

\def\CM{\operatorname{CM}}

\def\Cone{\operatorname{Cone}}

\def\D{\mathcal{D}}
\def\depth{\operatorname{depth}}
\def\E{\operatorname{E}}

\def\Ext{\operatorname{Ext}}

\def\ge{\geqslant}

\def\H{\operatorname{H}}

\def\hinf{\operatorname{hinf}}
\def\Hom{\operatorname{Hom}}
\def\hsup{\operatorname{hsup}}
\def\id{\mathrm{id}}
\def\Id{\mathrm{Id}}
\def\Im{\operatorname{Im}}
\def\K{\operatorname{K}}
\def\Ker{\operatorname{Ker}}

\def\le{\leqslant}
\def\len{\operatorname{length}}
\def\m{\mathfrak{m}}

\def\p{\mathfrak{p}}
\def\pd{\operatorname{pd}}

\def\qid{\operatorname{qid}}
\def\qpd{\operatorname{qpd}}

\def\rhom{\operatorname{RHom}}

\def\spec{\operatorname{Spec}}

\def\Tor{\operatorname{Tor}}

\def\xx{\boldsymbol{x}}
\def\Z{\mathbb{Z}}

\begin{document}
%\setlength{\baselineskip}{15pt}
%%%%%%%%%%%%%%%%%%%%%%%%%%%%%%%%%%%%%%%%%%%%%%%%%%%%%%%%
\title{Quasi-injective dimension}
\author{Mohsen Gheibi}
\address[M. Gheibi]{Department of Mathematics, Florida A{\&}M University, Tallahassee FL, USA }
\email{mohsen.gheibi@famu.edu}
\begin{abstract}
Following our previous work about quasi-projective dimension \cite{GJT}, in this paper, we introduce quasi-injective dimension as a generalization of injective dimension. We recover several well-known results about injective and Gorenstein-injective dimensions in the context of quasi-injective dimension such as the following. (a) If the quasi-injective dimension of a finitely generated module $M$ over a local ring $R$ is finite, then it is equal to the depth of $R$. (b) If there exists a finitely generated module of finite quasi-injective dimension and maximal Krull dimension, then $R$ is Cohen-Macaulay. (c) If there exists a nonzero finitely generated module with finite projective dimension and finite quasi-injective dimension, then $R$ is Gorenstein. (d) Over a Gorenstein local ring, the quasi-injective dimension of a finitely generated module is finite if and only if its quasi-projective dimension is finite. 
\end{abstract}
\maketitle
%%%%%%%%%%%%%%%%%%%%%%%%%%%%%%%%%%%%%%%%%%%%%%%%%%%%%%%%%
\section{Introduction}

In our previous work together with Jorgensen and Takahashi \cite{GJT}, we introduced and studied the quasi-projective dimension as a generalization of the projective dimension. Building upon this work, our current paper delves into the study of the quasi-injective dimension. Our motivation for defining these homological dimensions stems from the following phenomenon. 

Let $Q$ be a commutative Noetherian ring and let $f_1,\dots,f_n$ be a regular sequence in $Q$. Set $R=Q/(f_1,\dots,f_n)$ and let $M$ be an $R$-module, and $\dots \to F_1 \to F_0 \to 0$ be a projective  resolution of $M$ over $Q$. Then $F\otimes_QR$  is a complex of projective  $R$-modules with a special property that all nonzero homologies of $F\otimes_QR$  are isomorphic to finite direct sums of copies of $M$. We call a complex of projective  $R$-modules with such property a {\em quasi-projective resolution} of $M$. We say $M$ has finite {\em quasi-projective}  {\em dimension} over $R$ if there exists a bounded quasi-projective  resolution of $M$. In the analogous way, we define {\em quasi-injective resolution} and {\em quasi-injective dimension}; see Definition \ref{12}. These homological dimensions have been studied from a different perspective in \cite{Av2}. Also in \cite{DGI}, modules possessing finite quasi-projective dimension have been investigated as {\em virtually small objects} in the derived category of $R$; see also \cite{BGP} and \cite{P}.

Although quasi-resolutions are not always acyclic, they are still useful. Specifically, quasi-dimensions provide insights into both the modules themselves and the underlying ring.
In \cite{GJT}, we recovered the Auslander-Buchsbaum Formula  for finitely generated modules of finite quasi-projective dimension over a commutative Noetherian local ring $R$; see \cite[Theorems 4.4]{GJT}. On the other hand, if injective dimension of a nonzero finitely generated module $M$ is finite, then it is equal to the depth of $R$. This is known as the {\em Bass Formula}. In light of this, our first main result  is the following.

\begin{thm*}[A]
Let $R$ be a commutative Noetherian local ring and let $M$ be a nonzero finitely generated $R$-module of finite quasi-injective dimension. Then one has $\dim_RM\leq \qid_RM=\depth R$.
\end{thm*}
Bass \cite{B} conjectured that the existence of a nonzero finitely generated module $M$ of finite injective dimension forces $R$ to be Cohen-Macaualy which was proven later by Roberts \cite{Rob} as a consequence of Peskine and Szpiro's Intersection Theorem. A corollary of Theorem (A) is the following result that also can be compared to \cite[Theorems 3.4 and 3.5]{T}, \cite[Theorem 1.3]{Y}, and \cite[Corollary 7.6]{ZCGS}.
\begin{cor*}
Let $(R,\m)$ be a Noetherian local ring. If there exists a nonzero finitely generated module $M$ with maximal Krull dimension and finite quasi-injective dimension, then $R$ is Cohen-Macaulay. In particular, if $\qid_R\m<\infty$, then $R$ is Cohen-Macaulay.
\end{cor*}
In contrast to Bass's theorem, an arbitrary local ring $R$ also admits a finitely generated module of finite quasi-injective dimension. A straightforward example is the residue field of a maximal ideal of $R$ that has both finite quasi-projective and quasi-injective dimensions.

In light of the aforementioned corollary, a natural question arises: what happens if the quasi-injective dimension of $R$ is finite when considered as a module over itself?  This question is of broader interest, and in section 4, we prove the following result that addresses it and generalizes a theorem of Foxby \cite{Fox}.

\begin{thm*}[B]
Let $R$ be a commutative Noetherian local ring and let $M$ be a nonzero finitely generated $R$-module. If $\pd_RM$ and $\qid_RM$ both are finite, then $R$ is Gorenstein.
\end{thm*}
Theorem (B) is a corollary of a more general fact which is stated as Theorem \ref{Foxby}. 
The last section of this paper is devoted to showing that over a Gorenstein local ring $R$ the classes of finitely generated modules of finite quasi-projective  and finite quasi-injective dimensions coincide.

\begin{thm*}[C]
Let $R$ be a Gorenstein local ring and let $M$ be a finitely generated $R$-module. Then $\qpd_RM<\infty$ if and only if $\qid_RM<\infty$.
\end{thm*}

%%%%%%%%%%%%%%%%%%%%%%%%%%%%%%%%%%%%%%%%%%%%%%%%%%%%%%%%

%%%%%%%%%%%%%%%%%%%%%%%%%%%%%%%%%%%%%%%%%%%%%%%%%%%%%%%%

%%%%%%%%%%%%%%%%%%%%%%%%%%%%%%%%%%%%%%%%%%%%%%%%%%%%%%%%%
\section{Definitions and basic properties}

Throughout, all rings are commutative Noetherian with identity.

\begin{chunk}\label{notations}
Let $\A$ be an abelian category.
Let $$C= (\cdots\xrightarrow{\partial_{i+2}}C_{i+1}\xrightarrow{\partial_{i+1}}C_i \xrightarrow{\partial_{i}} C_{i-1} \xrightarrow{\partial_{i-1}} \cdots)$$ be a complex of objects of $\A$. We define the {\em supremum}, {\em infimum}, {\em homological supremum} and {\em homological infimum} of $C$ as follows
$$
\begin{cases}
\sup C=\sup\{i\in\Z\mid C_i\ne0\},\\
\inf C=\inf\{i\in\Z\mid C_i\ne0\},
\end{cases}
\qquad
\begin{cases}
\hsup C=\sup\{i\in\Z\mid\H_i(C)\ne0\},\\
\hinf C=\inf\{i\in\Z\mid\H_i(C)\ne0\}.
\end{cases}
$$
The \emph{length} of $C$ is defined to be $\len C=\sup C-\inf C$.
We say that $C$ is \emph{bounded}, if $\len C<\infty$. We say that $C$ is \emph{bounded below} if $\inf C>-\infty$ and $C$ is {\em bounded above} if $\sup C<\infty$. 
Note that if $C$ satisfies $C_i=0$ for all 
$i\in\mathbb Z$, then $\sup C =-\infty$, $\inf C=\infty$, and by convention we set $\len C=-\infty$.
For an integer $j$, the complex $C[j]$ is defined by ${C[j]}_i = C_{i-j}$ and $\partial^{C[j]}_i=(-1)^j\partial^C_{i-j}$ for all $i$.
\end{chunk}

\begin{defn}\label{12}
Let $\A$ be an abelian category with enough projective and injective objects.
Let $M$ be an object of $\A$.
\begin{enumerate}[(1)]
\item
A {\em quasi-projective resolution} of $M$ in $\A$ is a bounded below complex $P$ of projective objects in $\A$ such that for all $i\ge\inf P$ there exist non-negative integers 
$a_i$, not all zero, such that $\H_i(P)\cong M^{\oplus a_i}$.
We define the {\em quasi-projective dimension} of $M$ in $\A$ by
$$
\qpd_\A M=
\inf\{\sup P-\hsup P\mid\text{$P$ is a bounded quasi-projective resolution of $M$}\},
$$ and $\qpd_RM=-\infty$ if $M=0$.
\item 
A {\em quasi-injective resolution} of $M$ in $\A$ is a bounded above complex $I$ of injective objects in $\A$ such that for all $i\leq \sup I$ there exist non-negative integers $b_i$, not all zero, such that $\H_i(I)\cong M^{\oplus b_i}$. We define the {\em quasi-injective dimension} of  $M$ in $\A$ by
$$
\qid_\A M=
\inf\{\hinf I-\inf I\mid\text{$I$ is a bounded quasi-injective resolution of $M$}\},
$$ and $\qid_\A M=-\infty$ if $M=0$.
\end{enumerate}
One has $\qpd_\A M=\infty$ or $\qid_\A M=\infty$ if and only if $M$ does not admit a bounded quasi-projective or quasi-injective resolutions, respectively.
\end{defn}

\begin{rem}\label{8}
Let $\A$ be an abelian category with enough injective objects.
Let $M$ be an object of $\A$.
\begin{enumerate}[(1)]
\item
Every injective resolution of $M$ is a  quasi-injective resolution of $M$.
\item
If injective dimension of $M$ is finite, then quasi-injective dimension of $M$ is finite.
Therefore there is an inequality $\qid_\A M\le\id_\A M$.
\item Let $M\ne0$, and $I$ be a quasi-injective resolution of $M$.
Then there exists a quasi-injective resolution $I'$ of $M$ with $\H_i(I')=\H_i(I)$ for all 
$i\in\Z$, and $\hsup I'=\sup I'$. Therefore by taking shifts, we  will assume $\sup I=\hsup I=0$
\end{enumerate}
\end{rem}

By modifying the proof of \cite[Proposition 3.3]{GJT}, we have the following proposition.

\begin{prop}\label{Syz}
Let $\A$ be an abelian category with enough injective objects.
\begin{enumerate}[\rm(1)]
\item
For an object $M\in\A$ and an integer $n>0$, one has $\qid_\A(M^{\oplus n})=\qid_\A M$.
\item
Let $M,N\in\A$.
Then $\qid_\A(M\oplus N)\le\sup\{\qid_\A M,\qid_\A N\}$.
\item
For a nonzero object $M\in\A$ and an exact sequence $0\to N\to J \to M \to 0$ with $J$ an injective object, there is an inequality $\qid_\A M\le\qid_\A N$.
\end{enumerate}
\end{prop}

\begin{lem}\label{spseq}
Let $P$ be a bounded below complex of projective $R$-modules, and let $I$ be a bounded above complex of injective $R$-modules.
Then for an $R$-module $N$ there are convergent spectral sequences
$$\E^{p,q}_2= \Ext^p_R(\H_q(P),N) \Longrightarrow \H_{-p-q}(\Hom_R(P,N)),$$
$$\E^{p,q}_2= \Ext^p_R(N,\H_q(I)) \Longrightarrow \H_{q-p}(\Hom_R(N,I)).$$
\end{lem}
\begin{proof}
See for example, \cite[Theorem 11.34]{R}.
\end{proof}

\begin{prop}\label{7} Let $I$ be a quasi-injective resolution of an $R$-module $M$.
\begin{enumerate}[\rm(1)]
\item
If $S$ is a multiplicatively closed subset of $R$, then $I_S\cong I\otimes_R R_S$ is a quasi-injective resolution of $M_S$ over $R_S$. Moreover, $\qid_{R_S}M_S\le\qid_RM$.
\item
If $x \in R$ is a nonzero-divisor on both $R$ and $M$, then $\Hom_R(R/(x),I)$ is a quasi-injective resolution of the $R/(x)$-module $M/xM$. Moreover, one has $\qid_{R/(x)}M/xM\le\qid_RM$.
\item Suppose $R$ is local and $M$ is finitely generated. If $\qid_RM<\infty$, then quasi-injective dimension of $\widehat{M}$ is finite as an $R$-module, where $\widehat{M}=M\otimes_R\widehat{R}$ and $\widehat{R}$ is the completion of $R$ with respect to the maximal ideal $\m$. 
\end{enumerate}
\end{prop}
\begin{proof}
(1) We have $I\otimes_RR_S$ is a complex of injective $R_S$-modules; see \cite[Corollary 3.1.3 (a)]{BH}.
Since $R_S$ is flat, one has $\H_i(I\otimes_R R_S)\cong \H_i(I)\otimes_R R_S$, which finishes the proof.

(2) By Lemma \ref{spseq}, consider the spectral sequence
$$
\E^{p,q}_2= \Ext^p_R(R/(x),\H_q(I)) \Longrightarrow \H_{q-p}(\Hom_R(R/(x),I)).
$$
One has $\Hom_R(R/(x),I)$ is a complex of injective $R/(x)$-modules by \cite[Lemma 3.1.6]{BH}.  Since $x$ is a nonzero-divisor on both $R$ and $M$, we have 
$$\E^{p,q}_\infty\cong \begin{cases} \Ext^1_R(R/(x),\H_q(I)) & \text{if }\ p=1
\\
0 & \text{otherwise}. 
\end{cases}$$
Note that $\H_q(I)\cong \oplus^{b_q}M$ for some integers $b_q\geq 0$. Therefore we have 
\[\begin{array}{rl}\
\oplus^{b_q} M/xM &\cong \Hom_{R/(x)}(R/(x),\H_q(I)/x\H_q(I))\\
&\cong \Ext^1_R(R/(x),\H_q(I))\\
&\cong \H_{q-1}(\Hom_R(R/(x),I)).
\end{array}\]

(3) Assume $\qid_RM<\infty$ and $I$ is a bounded quasi-injective resolution of $M$. Set $(-)^{\vee}:=\Hom_R(-,\E(R/\m))$. Then $I^{\vee}$ is a bounded complex of flat $R$-modules, and therefore, it has a finite flat dimension in the derived category of $R$. By applying \cite[Theorem 8.3.18]{CFH}, one has $\id_RI^{\vee\vee}<\infty$. Hence there exists a semi-injective resolution $I^{\vee\vee}\overset{\simeq}\longrightarrow I'$, where $I'$ is a bounded complex of injective $R$-modules. Since homologies of $I^{\vee\vee}$ are isomorphic to direct sums of $M^{\vee\vee}\cong \widehat{M}$, we have $\qid_R\widehat{M}<\infty$. 
\end{proof}

\begin{rem}
In view of Proposition \ref{7}(3), we do not know if $\qid_RM$ and $\qid_{\widehat{R}}\widehat{M}$ coincide. However, if $R$ is Gorenstein, then one has $\qid_{\widehat{R}}\widehat{M} \le \qid_RM$; see Theorem \ref{main5}, \cite[Proposition 3.5(1)]{GJT}, and Theorem \ref{Bass}.
\end{rem}

The following proposition provides examples of modules with finite quasi-injective dimensions whose injective dimension may not be finite.

\begin{prop}\label{18} Let $R$ be a ring.
\begin{enumerate}[\rm(1)]
\item
If $\m$ is a maximal ideal of $R$ generated by $n$ elements, then the $R$-module $R/\m$ admits a quasi-injective resolution of length $n$. In particular, $\qid_R R/\m<\infty$
\item
Suppose $R=Q/(\xx)$ where $Q$ is a ring and $\xx=x_1,\dots,x_c$ a $Q$-regular sequence.
Let $M$ be an $R$-module, and $I$ be an injective resolution of $M$ over $Q$. Then $\Hom_Q(R,I)$ is a quasi-injective resolution of $M$ over $R$. In particular, if $\id_Q M<\infty$ then $\qid_R M<\infty$.
\end{enumerate}
\end{prop}

\begin{proof}
(1) Consider the the Koszul complex $\K(\xx)$ on a system of generators $\xx=x_1,\dots,x_n$ of  $\m$.
Then $\Hom_R(\K(\xx), \E(R/\m))$ is a bounded complex of injective $R$-modules where $\E(R/\m)$ is the injective envelope of $R/\m$. Then one has $$\H_{-i}(\Hom_R(\K(\xx), \E(R/\m))) \cong \Hom_R(\H_i(\xx),\E(R/\m)),$$ where the later is a finite dimensional vector space over $R/\m$. 

(2) Since $\xx$ is a regular sequence, the Koszul complex $\K(\xx,Q)$ is a free resolution of 
$R$ over $Q$. For each $0\le i\le c$ there are isomorphisms
\[\begin{array}{rl}\
\Ext_Q^i(R,M) &\cong \H_{-i}(\Hom_Q(\K(\xx,Q),M))\\
&\cong \H_{c-i}(\xx,M)\\
&\cong M^{\oplus\binom{c}{i}}
\end{array}\]
of $Q$-modules, where the second isomorphism comes from \cite[Propositions 1.6.9 and 1.6.10]{BH} and the last isomorphism holds since $\xx M=0$.
\end{proof}

As a corollary of Proposition \ref{18}(2), we obtain the following.

\begin{cor}\label{36}
Let $R$ be a quotient of a regular local ring by a regular sequence. Let $M$ be an $R$-module.
Then $\qid_RM<\infty$.
\end{cor}

We close this section by showing that quasi-dimensions remain finite when modding out by high enough powers of a regular sequence. First, we need the following lemma.

\begin{lem}\label{T41}
Let $Q$ be a ring and let $f\in Q$ be a nonzero-divisor. Let $M$ be a $Q/(f)$-module. Then the following hold.
\begin{enumerate}[\rm(1)]
    \item If $F$ is a quasi-projective resolution of $M$ over $Q$, then $\H_i(F\otimes_QQ/(f^n))\cong \H_i(F)\oplus \H_{i-1}(F)$ for all $n>i$. In particular, if $\hsup F<\infty$, then the complex $F\otimes_QQ/(f^n)$ is a quasi-projective resolution of $M$ over $Q/(f^n)$, for all $n > \hsup F$.
    \item If $I$ is a quasi-injective resolution of $M$ over $Q$, then $\H_i(\Hom_Q(Q/(f^n),I)) \cong \H_i(I)\oplus \H_{i+1}(I)$ for all $n>-i$.  In particular, if $\hinf I>-\infty$, then the complex $\Hom_Q(Q/(f^n),I)$ is a quasi-injective resolution of $M$ over $Q/(f^n)$ for all $n>-\hinf I$
\end{enumerate}
\end{lem}
\begin{proof}
We only prove (1), and the dual argument applies to (2). 
Let $$F=(\cdots \to F_2 \overset{\partial_2} \to F_1 \overset{\partial_1}\to F_0 \to 0)$$ be a quasi-projective resolution of $M$. 

{\bf Step 1.} For each $n\geq 1$ and $0\leq i <n$, we construct maps $\beta^n_i:F_i\to F_{i+1}$ such that $f^n\Id_{F_i}=\partial_{i+1}\beta^n_i+\beta^n_{i-1}\partial_i$. Set $B_i=\Im(\partial_{i+1})$ and $Z_i=\ker(\partial_i)$.
For $n=1$, since $\H_0(F) \overset{f} \to \H_0(F)$ is a zero map, $fF_0\subseteq B_1$. Therefore, there is a map $\beta^1_0:F_0\to F_1$ such that $f\Id_{F_0}=\partial_1\beta^1_0$. Let $n>1$ and assume we have constructed the maps $\beta^j_i:F_i\to F_{i+1}$ for $1\leq j \leq n$ and $0 \leq i < j$ such that $f^j\Id_{F_i}=\partial_{i+1}\beta^j_i+\beta^j_{i-1}\partial_i$.  
Then we have 
\[\begin{array}{rl}\
\partial_{n-1}(f^{n-1}\Id_{F_{n-1}}-\beta^{n-1}_{n-2}\partial_{n-1}) &= f^{n-1}\partial_{n-1}-\partial_{n-1}\beta^{n-1}_{n-2}\partial_{n-1}\\
&= f^{n-1}\partial_{n-1}-f^{n-1}\partial_{n-1}-\beta^{n-1}_{n-3}\partial_{n-2}\partial_{n-1}\\
&= 0.
\end{array}\]
It follows that for any $x\in F_{n-1}$ we have $y=f^{n-1}x-\beta^{n-1}_{n-2}\partial_{n-1}(x) \in Z_{n-1}$ and therefore $fy\in fZ_{n-1} \subseteq B_{n-1}$. Hence there exists a map $\beta^n_{n-1}:F_{n-1} \to F_n$ such that $$f^n\Id_{F_{n-1}}=\partial_n\beta^n_{n-1}+f\beta^{n-1}_{n-2}\partial_{n-1}.$$
By defining $\beta^n_i:=f\beta^{n-1}_i$ for $0\leq i< n-1$, we are done.

{\bf Step 2.} For each $i\geq 0$, we show that $\H_i(\Cone(f^n))\cong \H_{i-1}(F)\oplus \H_{i}(F)$ for all $n>i$ where $\Cone(f^n)$ is the mapping cone of $F\overset{f^n}\to F$. Fix $i\geq 0$ and $n>i$. By using Step 1, for $0\leq j\leq i$ we define the maps
$$
\begin{CD}
\alpha_j: \Cone(f^n)_j @>{\left(\begin{array}{ccc}
1_{F_{j-1}} & 0  \\
\beta^n_{j-1} & 1_{F_j} 
\end{array}\right)} >>  (F[-1]\oplus F)_j.
\end{CD}$$
Since differentials of $\Cone(f^n)$ and $F[-1] \oplus F$ respectively are $\left(\begin{array}{ccc}
-\partial_{j-1} & 0  \\
f^n & \partial_j 
\end{array}\right)$ and 
$\left(\begin{array}{ccc}
-\partial_{j-1} & 0  \\
0 & \partial_j
\end{array}\right)$, one can easily check that $\alpha_j$ is a map of complexes for $0\leq j \leq i$.  Since $f\H_i(F)=0$ for all $i$, there is a commutative diagram with exact rows
$$
\begin{CD}
0 @>>> \H_i(F)  @>>> \H_i(\Cone(f^n)) @>>> \H_{i-1}(F) @>>> 0 \\
@. @VV=V @VVV @VV=V \\
0 @>>> \H_i(F) @>>> \H_{i-1}(F) \oplus \H_i(F) @>>> \H_{i-1}(F) @>>> 0,
\end{CD}
$$ 
where the middle map is induced by $\alpha_i$. Hence by the Five Lemma, we get the isomorphism $\H_i(\Cone{f^n})\cong \H_{i-1}(F) \oplus \H_i(F)$.

Note that for all $i\geq 0$ and $n\geq 1$, we have $\H_i(F\otimes_QQ/(f^n))\cong \H_i(F\otimes_Q\K(f^n))\cong \H_i(\Cone(f^n))$ where $\K(f^n)$ is the Koszul complex. This finishes the proof.
\end{proof}

\begin{prop}\label{C51}
Let $R=Q/(f)$ where $f\in Q$ is a nonzero-divisor. Let $M$ be an $R$-module.
\begin{enumerate}[\rm(1)]
    \item If $\qpd_QM<\infty$ then $\qpd_{Q/(f^n)}M<\infty$ for all $n\gg 0$. Conversely, if  $\qpd_RM<\infty$ then $\qpd_QM<\infty$.
    \item If $\qid_QM<\infty$ then $\qid_{Q/(f^n)}M<\infty$ for all $n\gg 0$. Conversely, if  $\qid_RM<\infty$ then $\qid_QM<\infty$.
\end{enumerate}
\end{prop}
\begin{proof}
(1) It follows from Lemma \ref{T41} and \cite[Proposition 3.5(3)]{GJT}.

(2) If $\qid_QM<\infty$, then by Lemma \ref{T41}, $\qid_{Q/(f^n)}M<\infty$ for all $n\gg 0$. 
Suppose $\qid_RM<\infty$. First we show that if $E$ is an injective $R$-module, then $\id_QE<\infty$. To this end, by using \cite[Corollary 3.1.12]{BH} we show $\Ext^2_Q(Q/\p,E)=0$ for all $\p \in \spec Q$. Let $\p \in \spec Q$. If $f \notin \p$, then the exact sequence $0\to Q/\p \overset{f}\to Q/\p \to Q/(f,\p) \to 0$ induces an exact sequence 
$$0\to \Ext^2_Q(Q/\p,E) \to \Ext^2_Q(Q/(f,\p),E) \to \Ext^3_Q(Q/\p,E)\to 0.$$ Therefore if we show $\Ext^2_Q(N,E)=0$ for every finitely generated $R$-module $N$, we are done. Consider the change of ring spectral sequence 
$$\Ext^p_R(N,\Ext^q_Q(R,E))\Longrightarrow \Ext^{p+q}_Q(N,E).$$
Since $f$ is a regular element and $E$ is an $R$-module, we have $\Ext^q_Q(R,E)\cong E$ for $q=0,1$ and zero otherwise. Therefore $\Ext^p_R(N,\Ext^q_Q(R,E))=0$ for all $p>0$ and $q>1$. This implies that $\Ext^n_Q(N,E)=0$ for all $n\geq 2$. 

Now if $\qid_RM<\infty$, then there exists a bounded quasi-injective resolution $I$ of $M$ over $R$. Since by \cite[Example 3.2]{DGI}, the subcategory of complexes of finite injective dimension is thick in $\D(Q)$, we see by induction that $\id_{\D(Q)}I<\infty$. Therefore there exists a bounded semi-injective complex $I'$ over $Q$ and a quasi-isomorphism $I \overset{\simeq}\longrightarrow I'$; see \cite[Theorem 8.2.8]{CFH}. Therefore $\qid_QM<\infty$.
\end{proof}

\begin{ques}
Let $R=Q/(f)$ where $f$ is a nonzero-divisor in $Q$, and let $M$ be an $R$-module. Is it true that if $\qpd_QM<\infty$ ($\qid_QM<\infty$) then $\qpd_RM<\infty$ ($\qid_RM<\infty$)?
\end{ques}

%%%%%%%%%%%%%%%%%%%%%%%%%%%%%%%%%%%%%%%%
\section{Generalization of a theorem of Bass}

In this section, we establish a version of the Bass Formula and generalize a theorem of Bass for finitely generated modules of finite quasi-injective dimension. 

\begin{lem}\label{L31} Let $M$ be a finitely generated $R$-module and let 
$$I=(0\longrightarrow I_0 \overset{\partial_0} \longrightarrow I_{-1} \overset{\partial_{-1}} \longrightarrow I_{-2} \longrightarrow \dots )$$ be a quasi-injective resolution of $M$. Set $B_i =\Im \partial_i$ and $Z_i=\Ker \partial_i$. Then for every finitely generated $R$-module $N$, $\Ext^j_R(N,B_i)$ and $\Ext^j_R(N,Z_i)$ are finitely generated $R$-modules, for all $i\in \Z$ and all $j>0$.
\end{lem}
\begin{proof} Consider the exact sequences 
$$\begin{cases}
0\to Z_i\to I_i\to B_i\to0\\
0\to B_{i+1}\to Z_i\to \H_i(I)\to0
\end{cases}
(i\in\Z).$$
Since $Z_0\cong \H_0(I) \cong \oplus^{b_0}M$ for some $b_0\geq 0$, we have $\Ext^j_R(N,Z_0)$ is finitely generated for all $j$. From the first exact sequence above, we get $\Ext^{j}_R(N,B_0)\cong \Ext^{j+1}_R(N,Z_0)$ for all $j>0$. This proves the claim for $\Ext^{j}_R(N,B_0)$. 

Next assume $i<0$ and the assertion holds for $B_r$ and $Z_r$, $r=0,-1,\dots,i+1$. From the second exact sequence above, we obtain an exact sequence
$$\Ext^j_R(N,B_{i+1}) \to \Ext^j_R(N,Z_i) \to \Ext^j_R(N,\H_i(I)).$$
By assumptions and induction, we see that $\Ext^j_R(N,Z_i)$ is finitely generated for all $j>0$. Finally since $\Ext^j_R(N,B_i)\cong \Ext^{j+1}_R(N,Z_i)$, we have $\Ext^j_R(N,B_i)$ is finitely generated for all $j>0$.
\end{proof}

\begin{thm}\label{Bass}
Let $(R,\m,k)$ be a local ring and let $M$ be a nonzero finitely generated $R$-module. If $\qid_RM<\infty$ then $\qid_RM=\depth R$.
\end{thm}
\begin{proof}
Since $\qid_RM<\infty$, there exists a quasi-injective resolution $I$ of $M$ such that $\qid_RM=\hinf I-\inf I$.
Set $s=\hinf I$, $t=\qid_RM$, and $d=\depth R$. Using the notation in Lemma \ref{L31} and by choice of $I$, we have $\id_RZ_s=t$. For a prime ideal $\p\neq \m$ of $R$ and $x\in \m \setminus \p$, consider the exact sequence $0\longrightarrow R/\p \overset{x} \longrightarrow R/\p$. This yields an exact sequence 
$$\Ext^t_R(R/\p,Z_s) \overset{x} \longrightarrow \Ext^t_R(R/\p,Z_s)\longrightarrow 0.$$ By Lemma \ref{L31} and Nakayama's Lemma, $\Ext^t_R(R/\p,Z_s)=0$. Since $\id_RZ_s=t$, we must have $\Ext^t_R(k,Z_s)\neq 0$; see \cite[Corollary 3.1.12]{BH}. Let $\xx=x_1,\dots,x_d$ be a maximal regular sequence in $R$. Then there exists an inclusion $0\to k \to R/(\xx)$. Hence we get a surjection $\Ext^t_R(R/(\xx),Z_s)\to \Ext^t_R(k,Z_s) \to 0$ which implies $\Ext^t_R(R/(\xx),Z_s)\neq 0$. Since $\pd_RR/(\xx)=d$, we get $t\leq d$.

Next, we show $t\geq d$. By the exact sequence $0\to B_{s+1}\to Z_s \to \H_s(I)\to 0$, we get the surjection $\Ext^d_R(R/(\xx),Z_s)\to \Ext^d_R(R/(\xx),\H_s(I))\to 0$. Since $\H_s(I)$ is nonzero and finitely generated, we have $\Ext^d_R(R/(\xx),\H_s(I)) \cong \H_0(\K(\xx,\H_s(I))) \cong \H_s(I)/\xx\H_s(I)\neq 0$; see \cite[Proposition 1.6.10]{BH}. Hence $\Ext^d_R(R/(\xx),Z_s)\neq 0$ and so that $t\geq d$.
\end{proof}

\begin{cor} Let $R$ be a local ring and let $M$ be a finitely generated $R$-module. Then one has $\qid_RM\leq \id_R M$ and if $\id_R M<\infty$, then equality holds.
\end{cor}

Next, we prove a result regarding the vanishing of Ext and local cohomology for modules of finite quasi-injective dimension.

\begin{prop}\label{main1}
Let $M,N$ be $R$-modules and assume $\qid_RN<\infty$. 
Let $I$ be a quasi-injective resolution of $N$ such that $\qid_RN=\hinf I - \inf I$.
Then the following hold.
\begin{enumerate}[\rm(1)]
\item If {$\Ext^i_R(M,N)=0$ for all $n\leq i \leq n-\inf I$} and some $n\geq 1$, then {$\Ext^i_R(M,N)=0$} for all $i\geq n$.
\item If {$\Ext^{\gg 0}_R(M,N)=0$}, then {$\Ext^i_R(M,N)=0$} for all $i>\qid_RN$.
\item For an ideal $\aa$ of $R$ one has $\H^i_{\aa}(N)=0$ for all $i>\qid_RN$.
\end{enumerate}
\end{prop}
\begin{proof}  Set $d=\qid_RN$, $s=\hinf I$, and $r=\inf I$. By using Lemma \ref{spseq}, there exists a fourth quadrant spectral sequence 
$$\E^{p,q}_2=\Ext^p_R(M,\H_q(I)) \Longrightarrow \H_{q-p}(\Hom_R(M,I)).$$
(1) By induction, it is enough to show $\Ext^{n-r+1}_R(M,N)=0$. Note that $\E^{p,q}_2=0$ for all $q<\hinf I$ and
the maps $d^{p,q}_l:\E^{p,q}_l \to \E^{p+l,q+l-1}_l$ are of bidegree $(l,l-1)$. By the assumptions, we see that no nonzero map arrives or goes out from $\E^{n-r+1,0}_2$. Thus $\E^{n-r+1,0}_2\cong \E^{n-r+1,0}_\infty$, and so it is a subquotient of $\H_{-n+r-1}(\Hom_R(M,I))$. Hence $\H_{-n+r-1}(\Hom_R(M,I))=0$ as $-n+r-1<r=\inf I$.

(2) Let $j\geq 1$ be the minimum integer such that $\Ext^j_R(M,N)\neq 0$. Then we have $\E^{j,s}_2\cong \Ext^j_R(M,\H_s(I))\neq 0$ and since $d^{j,s}_l$ has bidegree $(l,l-1)$ for all $l\geq 2$, one has $\E^{j,s}_2\cong \E^{j,s}_\infty$. Therefore $\E^{j,s}_2$ is a subquotient of $\H_{s-j}(I)$. Thus we must have $s-j\geq r$ and so $j\leq s-r=\qid_R N$.

(3) Suppose $\aa$ can be generated by $t$ elements. If $d\geq t$, then we have nothing to prove; see \cite[Theorem 3.3.1]{BS}. Assume $d < t$. Let $d<j\leq t$ and suppose $\H^i_{\aa}(N)=0$ for all $i>j$. By using the notation in Lemma \ref{L31}, we have an exact sequence $0\to B_{s+1}\to Z_s \to \H_s(I) \to 0$ such that $\id_R{Z_s} = d$. By applying $\Gamma_{\aa}(-)$ we get an exact sequence 
$$\H^j_{\aa}(Z_s) \to \H^j_{\aa}(\H_s(I))\to \H^{j+1}_{\aa}(B_{s+1})\to \H^{j+1}_{\aa}(Z_s).$$
As $\id_RZ_s < j$, we have $\H^j_{\aa}(Z_s)=0=\H^{j+1}_{\aa}(Z_s)$ and so $\H^j_{\aa}(\H_s(I))\cong \H^{j+1}_{\aa}(B_{s+1})$. Next by using the exact sequence $$0\to Z_{s+1} \to I_{s+1} \to B_{s+1}\to 0,$$ we have $\H^{j+1}_{\aa}(B_{s+1})\cong \H^{j+2}_{\aa}(Z_{s+1})$, and by using the exact sequence $$0\to B_{s+2}\to Z_{s+1} \to \H_{s+1}(I) \to 0,$$ we have $\H^{j+2}_{\aa}(B_{s+2})\cong \H^{j+2}_{\aa}(Z_{s+1})$. By continuing this way, we eventually get $$\H^j_{\aa}(\H_s(I))\cong \H^{j+1-s}_{\aa}(Z_0)\cong \H^{j+1-s}_{\aa}(\H_0(I)).$$
Since $s\leq 0$ and $\H_s(I)$, $\H_0(I)$ are isomorphic to finite direct sums of copies of $N$, last isomorphisms show that $\H^j_{\aa}(N)=0$.
\end{proof}

Putting Theorem \ref{Bass} together with Proposition \ref{main1}(2), we have the following corollary.

\begin{cor} \label{C37} Let $R$ be a local ring and let $M$ and $N$ be $R$-modules such that $N$ is finitely generated and $\qid_RN<\infty$. If $\Ext^{\gg0}_R(M,N)=0$, then $\Ext^i_R(M,N)=0$ for all $i>\depth R.$
\end{cor}

Let $R$ be a ring, $\aa$ an ideal of $R$, and let $M$ be an $R$-module. In \cite{Saz}, it is shown that the $i$th local cohomology module $\H^i_{\aa}(M)$ is obtained by applying $\Gamma_{\aa}(-)$ to a Gorenstein-injective resolution of $M$. In light of this result, our next corollary shows that $\H^i_{\aa}(M)$ also can be calculated by using a resolution of $M$ consisting of modules with quasi-injective dimension zero.

\begin{cor}
Let $R$ be a ring and let $\aa$ be an ideal of $R$. Let $M$ be an $R$-module and suppose there exists an exact sequence of $R$-modules
$$0\to M \to X_{0} \to X_{-1} \to \cdots,$$
where $\qid_RX_j=0$ for all $j$. Let $X=(0\to  X_{0} \to X_{-1} \to \cdots)$. Then one has $\H^i_{\aa}(M)\cong \H_{-i}(\Gamma_{\aa}(X))$ for all $i\geq 0$.
\end{cor}
\begin{proof}
Note that for an $R$-module $Y$ with $\qid_RY=0$, one has $\H^i_{\aa}(Y)=0$ for all $i>0$, by Proposition \ref{main1}(3). Therefore the claim follows by \cite[Exercise 4.1.2]{BS}.
\end{proof}

\begin{thm}\label{dim}
Let $(R,\m)$ be a local ring and let $M$ be a finitely generated $R$-module. If $\qid_RM<\infty$, then $\dim_RM\leq \depth R$.
\end{thm}
\begin{proof}
If $M=0$, then we have nothing to prove, so assume $M$ is nonzero. Then by Theorem \ref{Bass}, we have $\qid_RM=\depth R$. Thus by Proposition \ref{main1}(3), $\H^i_\m(M)=0$ for all $i>\depth R$. By applying Grothendieck's Non-Vanishing Theorem, we are done; see for example, \cite[Theorem 3.5.7]{BH}.
\end{proof}

The following corollary is an immediate consequence of Theorems \ref{Bass} and \ref{dim}.

\begin{cor}\label{35}
Let $(R,\m)$ be a local ring. If there exists a finiteley generated $R$-module $M$ such that $\dim_RM=\dim R$ and $\qid_RM<\infty$, then $R$ is Cohen-Macaulay. In particular if $\qid_R\m<\infty$, then $R$ is Cohen-Macaulay.
\end{cor}

 We remark that the finiteness of $\qid_R\m$ does not necessarily imply that $R$ is Gorenstein. For example, let $(R,\m)$ be a local ring such that $\m^2=0\neq \m$. Then $\m$ is isomorphic to a finite direct sum of copies of $k=R/\m$. Therefore $\qid_R\m<\infty$, but $R$ is not necessarily Gorenstein. In general, since $\m_{\p}\cong R_{\p}$ for all $\p \in \spec(R) \setminus \{\m \}$, by using Proposition \ref{7}(1) and Corollary \ref{C42}, we see that if $\qid_R\m<\infty$, then $R_\p$ is Gorenstein for all $\p \in \spec(R) \setminus \{\m \}$.

The example below shows that the inequality in Proposition \ref{7}(2) is indeed strict. 

\begin{exam}
Let $k$ be a field and $R=k[[x,y,z]]/(y^2,yz,z^2)$. Then $R$ is a Cohen-Macaulay local ring of dimension one. Let $\p=(y,z)$ and $\m=(x,y,z)$. Since $R_\p$ is not Gorenstein, we have $\qid_R\m=\infty$, by the last argument. One checks $x$ is a nonzero-divisor on both $R$ and $\m$, and $\m/x\m \cong k^3$. Hence $\qid_{R/(x)}\m/x\m < \infty$.  
\end{exam}
%%%%%%%%%%%%%%%%%%%%%%%%%%%%%%%%%%%%%%%%%%%%%
\section{Generalization of a theorem of Foxby}

In this section, we generalize a Theorem of Foxby \cite{Fox}, which states that if a local ring admits a nonzero finitely generated module with both finite injective and projective dimensions, then the ring is Gorenstein. First, we prove the following more general result.

\begin{thm}\label{Foxby} Let $(R,\m,k)$ be a local ring and let $M$ be a nonzero finitely generated $R$-module of finite projective dimension. Then the following statements are equivalent. 
\begin{enumerate}[\rm(1)]
    \item $R$ is Gorenstein.
    \item There exists an $R$-module $N$ of finite injective dimension and a surjective map $N \twoheadrightarrow M$.
\end{enumerate}
\end{thm} 
\begin{proof} (1)$\Rightarrow$(2) This is clear based on the fact that if $R$ is Gorenstein, then injective dimension of $R$ is finite.

(2)$\Rightarrow$(1) Let $F$ be a minimal free cover of $M$. Then there exists a commutative diagram 
\begin{center}
    \begin{tikzpicture}[baseline=(current  bounding  box.center)]
 \matrix (m) [matrix of math nodes,row sep=3em,column sep=4em,minimum width=2em] {
N&M\\
&F\\};
 \path[->>] (m-1-1) edge node[above]{} (m-1-2);
 \path[<<-] (m-1-2) edge node[right]{} (m-2-2);
 \path[<-] (m-1-1) edge node[below]{} (m-2-2);
 \end{tikzpicture}.
\end{center} 
 It follows that the induced map $\Ext^i_R(k,F)\to \Ext^i_R(k,M)$ is zero for all $i>\id_RN$.
Then by using the exact sequence $0\to \Omega M \to F \to M \to 0$, we obtain an exact sequence 
$$0\to \Ext^{n-1}_R(k,M) \to \Ext^n_R(k,\Omega M)\to \Ext^n_R(k,F)\to 0,$$
for all $n\gg 0$. This gives an equality of Bass numbers $$\mu^n_R(F)=\mu^n_R(\Omega M)-\mu^{n-1}_R(M),$$
for all $n\gg0$.
Since $\pd_RM<\infty$, by \cite[Proposition 2.2 and Theorem 6.1]{AVE} there exist isomorphisms of $k$-vector spaces
$$\Ext^{n-1}(k,M)\cong \bigoplus_{i-j=n-1}\Ext^i(k,R)\otimes_k\Tor^R_j(k,M)$$
and 
$$\Ext^n(k,\Omega M)\cong \bigoplus_{i-j=n}\Ext^i(k,R)\otimes_k\Tor^R_j(k,\Omega M).$$
Set $d=\pd_RM$. The last isomorphisms give the following equations:
$$\mu^{n-1}_R(M)=\mu^{n-1}_R(R)\beta_0(M)+\mu^n_R(R)\beta_1(M)+\cdots+
\mu^{n-1+d}_R(R)\beta_d(M)$$and
$$\mu^n_R(\Omega M)=\mu^n_R(R)\beta_0(\Omega M)+\mu^{n+1}_R(R)\beta_1(\Omega M)+\cdots+\mu^{n+d-1}_R(R)\beta_{d-1}(\Omega M).$$
Note that $\pd_R\Omega M=d-1$ and $\beta_i(M)=\beta_{i-1}(\Omega M)$. Therefore for all $n\gg 0$, we have
$$\mu^n_R(F)=\mu^n_R(\Omega M)-\mu^{n-1}_R(M)=-\mu^{n-1}_R(R)\beta_0(M) \leq 0,$$ This implies that $\mu^n_R(R)=0$ for all $n\gg 0$ and therefore $R$ is Gorenstein.
\end{proof}

\begin{rem}
The dual version of Theorem \ref{Foxby} is not true in general. That is, if there exists a finitely generated $R$-module $M$ of finite injective dimension and an injection $M\hookrightarrow N$ with $\pd_RN<\infty$, then $R$ is not necessarily Gorenstein. For example, if $R$ is a Cohen-Macaulay non-Gorenstein domain with a dualizing module $\omega_R$, then it is known that $\omega_R$ is an ideal of $R$ so there is an injection $\omega_R \hookrightarrow R$; see \cite[Proposition 3.3.18]{BH}. 
\end{rem}

\begin{cor}\label{C42}
Let $R$ be a local ring and let $M$ be a nonzero finitely generated $R$-module. Assume $\pd_RM<\infty$ and $\qid_RM<\infty$. Then $R$ is Gorenstein.
\end{cor}
\begin{proof} Let $I$ be a bounded quasi-injective resolution of $M$, and let $s=\hinf I$. Then by using the notation in Lemma \ref{L31}, $\id_RZ_s<\infty$. Since there is a surjection $Z_s \twoheadrightarrow \H_s(I)$ and $\H_s(I)$ is isomorphic to the finite direct sum of copies of $M$, by using Theorem \ref{Foxby}, $R$ is Gorenstein.
\end{proof}

\begin{rem} It is worth pointing out that if injective and quasi-projective dimensions of a nonzero finitely generated $R$-module  $M$ both are finite, then $R$ is Gorenstein; see \cite[Proposition 3.11]{GJT} and \cite[Theorem 5.3]{DGI}.
\end{rem}

A Gorenstein local ring $R$ is called an {\em AB ring} if for every finitely generated modules $M$ and $N$ whenever $\Ext^{\gg 0}_R(M,N)=0$ then $\Ext^i_R(M,N)=0$ for all $i>\depth R$. 

\begin{cor}
Let $R$ be a local ring. If quasi-injective dimension of every finitely generated $R$-module is finite, then $R$ is an AB ring.
\end{cor}
\begin{proof}
It follows from Corollaries \ref{C37} and \ref{C42}.
\end{proof}
By using Theorem \ref{main5} and \cite[Example 6.6]{GJT}, there exists an AB ring $R$ and a finitely generated $R$-module $M$ such that $\qid_RM=\infty$.

The following is a dual version of \cite[Theorem 6.20]{GJT} in the sense of quasi-injective dimension.
\begin{thm}
Let $R$ be a ring and let $M$ be an $R$-module of finite quasi-injective dimension. If $\Ext^{i>0}_R(M,M)=0$ then $\id_RM<\infty$. 
\end{thm}
\begin{proof}
Let $I$ be a bounded quasi-injective resolution of $M$, and adopt the notation in Lemma \ref{L31}. First, we show that $\Ext^{i>0}_R(M,B_j)=0$ and $\Ext^{i>0}_R(M,Z_j)=0$ for all $j\geq 0$. 
By using the exact sequence 
$0\to Z_0 \to I_0 \to B_0 \to 0$ and noting that $Z_0\cong \oplus^{b_0}M$ for some positive integer $b_0$, we see that $\Ext^{i>0}(M,B_0)=0$.
Next assume $j>0$ and consider the exact sequence $0\to B_{j+1}\to Z_j \to \H_j(I) \to 0$. It follows by induction and assumptions that $\Ext^{i>0}_R(M,Z_j)=0$. Finally, the exact sequence $0\to Z_j \to I_j \to B_j \to 0$ shows that $\Ext^{i>0}_R(M,B_j)=0$. 

Now let $s=\hinf I$ and consider the exact sequence $$0 \to B_{s+1} \to Z_s \to \H_s(I) \to 0.$$
Since $\H_s(I)\cong \oplus^{b_s}M$ for some $b_s>0$ and $\Ext^1_R(M,B_{s+1})=0$, the last exact sequence splits. Since $\id_RZ_s<\infty$, one has $\id_R\H_s(I)<\infty$. Thus, $\id_RM<\infty$.
\end{proof}

A finitely generated $R$-module $C$ is called {\em semidualizing} if the homothety map $R\to \rhom_R(C,C)$ is a quasi-isomorphism.

\begin{cor}
Let $R$ be a local ring and let $C$ be a semidualizing $R$-module. If $\qid_RC<\infty$ then $C$ is the dualizing module.
\end{cor}

\section{Quasi dimensions and duality}

In this section, we show that over a Gorenstein local ring, quasi-dimensions of a finitely generated module either both are finite or infinite. This result is well-known for projective and injective dimensions.

Let $\CM^n(R)$ be the category of finitely generated Cohen-Macaulay $R$-modules of dimension $n$.

\begin{thm}\label{duality}
Let $R$ be a Cohen-Macaulay local ring of dimension $d$ admitting a dualizing module $\omega_R$. Then the functor $$\Ext^{d-n}_R(-,\omega_R): \CM^n(R) \to \CM^n(R)$$ induces an equivalence between the categories of $n$-dimensional Cohen-Mcaulay modules of finite quasi-projective dimension and finite quasi-injective dimension. 
\end{thm}
\begin{proof}
Let $M$ be a finitely generated Cohen-Macaulay $R$-module of dimension $n$. Since $M$ is Cohen-Macaulay, by using Proposition \ref{C51}, we may choose an appropriate regular sequence contained in $\ann_R(M)$ and assume that $M$ is maximal Cohen-Macaulay.

If $\qpd_RM<\infty$, then there exists a bounded quasi-free resolution
$$F=(0\to F_l \to \cdots \to F_1 \to F_0 \to 0)$$
of $M$ such that $F_i$ is a finitely generated free module for each $i$; see \cite[Proposition 3.4]{GJT}. Let $D(R)$ be a dualizing complex of $R$. Then $\Hom_R(F,D(R))$ is a bounded complex of injective modules and by Lemma \ref{spseq} and \cite[Theorem 3.3.10]{BH} we have $\H_i(\Hom_R(F,D(R)))\cong  \Hom_R(\H_{-i}(F), \omega_R)$. This shows that $\qid_R\Hom_R(M,\omega_R)<\infty$. 

Next, assume $\qid_RM<\infty$ and let $I$ be a bounded quasi-injective resolution of $M$. Then $\Hom_R(I,D(R))$ is a bounded complex of flat $R$-modules whose nonzero homologies are isomorphic to finite direct sums of copies of $\Hom_R(M,\omega_R)$. By applying \cite[Theorem 8.3.19]{CFH}, we get $\pd_R\Hom(I,D(R))<\infty$ in the derived category of $R$. Thus there exists a bounded semi-projective complex $P$ and a quasi-isomorphism $P\overset{\simeq}\longrightarrow \Hom_R(I,D(R))$. This shows that $\qpd_R\Hom_R(M,\omega_R)<\infty$.
\end{proof}

\begin{prop}\label{P52}
Let $R$ be a ring and let $0\to N \to X \to M \to 0$ be an exact sequence of $R$-modules.
\begin{enumerate}[\rm(1)]
    \item If $\qpd_RM<\infty$ and $\pd_RX<\infty$, then $\qpd_RN<\infty$.
    \item If $\qid_RN<\infty$ and $\id_RX<\infty$, then $\qid_RM<\infty$.
\end{enumerate}
\end{prop}
\begin{proof}
We only prove (2), and a dual argument applies to (1). Assume $\qid_RN<\infty$ and let $$I=(0\to I_0 \to I_{-1} \to \cdots \to I_l \to 0)$$ be a bounded quasi-injective resolution of $N$ such that $\H_i(I)\cong \oplus^{n_i}N$.
Let $(E, \partial)$ be an injective resolution of $X$. Let $(J_{ij},d_{ij},d'_{ij})$ be a double complex such that $J_{ij}=\oplus^{n_i}E_j$, $d_{ij}=\partial_j^{n_i}:J_{ij}\to J_{ij-1}$, and $d'_{ij}:J_{ij}\to J_{i-1j}$ is the zero map. By using the notation in Lemma \ref{L31}, for each $i$ the diagram 
$$
\begin{CD}
0 @>>> Z_i  @>>> I_i\\
@. @VVV \\
@. \H_i(I)\\
@. @V h_i VV \\
@. \oplus^{n_i}E_0
\end{CD}
$$ induces a map $\alpha_i:I_i \to \oplus^{n_i}E_0$ such that the completed diagram is commutative, where $h_i$ is the composition of the incusions $\H_i(I) \hookrightarrow \oplus^{n_i}X \hookrightarrow \oplus^{n_i}E_0 $. Let $J$ be the total complex of $(J_{ij},d_{ij},d'_{ij})$. Then one checks that $\H_i(J)\cong X^{n_i}$ and  $\alpha_i$ induces a map of complexes $\beta:I\to J$. Then the exact sequence $$0\to J \to \Cone(\beta) \to \Sigma I \to 0$$ induces an exact sequence 
$$\cdots \to \H_i(J)\to \H_i(\Cone(\beta)) \to \H_{i-1}(I) \to \H_{i-1}(J) \to \cdots,$$
where the connecting homomorphism is the inclusion $\H_i(I) \hookrightarrow \oplus^{n_i}X$. Therefore we have an exact sequence $0\to \H_i(I)\to \H_i(J) \to \H_i(\Cone(\beta)) \to 0$ which shows that $\H_i(\Cone(\beta))\cong \oplus^{n_i}M$. Since $\Cone(\beta)$ is a bounded complex of injective modules, we have $\qid_RM<\infty$.
\end{proof}

Now we  prove the main result of this section.

\begin{thm}\label{main5}
Let $(R,\m)$ be a Gorenstein local ring and let $M$ be a finitely generated $R$-module. Then $\qid_RM<\infty$ if and only if $\qpd_RM<\infty$
\end{thm}
\begin{proof}
First, we assume $M$ is maximal Cohen-Macaulay. Then by Theorem \ref{duality}, $\qid_RM<\infty$ if and only if $\qpd_R\Hom_R(M,R) < \infty$. Since $M$ is totally reflexive, by \cite[Proposition 6.14]{GJT} we have $\qpd_R \Hom_R(M,R)<\infty$ if and only if $\qpd_RM<\infty$. 

Next, assume $M$ is any module. If $\qpd_RM<\infty$, then by using Proposition \ref{P52}(1), we have $\qpd_R\Omega^nM<\infty$ for all $n\geq 0$. By choosing $n$ large enough, we have $\Omega^nM$ is maximal Cohen-Macaulay and $\qpd_R\Omega^nM<\infty$. Therefore $\qid_R\Omega^nM<\infty$. Since $\id_RR<\infty$, by Proposition \ref{P52}(2), we get $\qid_RM<\infty$.

Now assume $\qid_RM<\infty$. By \cite[Theorem A]{ABU}, there exists an exact sequence $$0\to M \to X \to N \to 0,$$ where $\id_RX<\infty$ and $N$ is maximal Cohen-Mcaulay. Then by Proposition \ref{P52}(2), we have $\qid_RN<\infty$ and so $\qpd_RN<\infty$. Since $\id_RX<\infty$ and $R$ is Gorenstein, $\pd_R X<\infty$. Therefore $\qpd_RM<\infty$ by Proposition \ref{P52}(1).
\end{proof}

\begin{cor}
Let $R$ be a Gorenstein local ring and let $M$ and $N$ be finitely generated $R$-modules with $\qid_RM<\infty$. Then one has $\Ext^{\gg 0}_R(M,N)=0$ if and only if $\Ext^{\gg 0}_R(N,M)=0$. 
\end{cor}
\begin{proof}
It follows from Theorem \ref{main5} and \cite[Theorem 6.16]{GJT}.
\end{proof}

\begin{ac}
The author gratefully acknowledges the insightful conversation with Majid Rahro Zargar during the manuscript preparation. We thank Amir Mafi for bringing \cite{Saz} to our attention. We extend our appreciation to the anonymous referee for providing valuable comments and suggestions, which greatly contributed to the improvement of this paper. 
\end{ac}

%%%%%%%%%%%%%%%%%%%%%%%%%%%%%%%%%%%%%%%%%%%%%%%%%%%%%%%%
\end{document}